\documentclass[journal,twoside,web]{ieeecolor}
\usepackage{generic}
\usepackage{textcomp}

\usepackage{cite}\usepackage{hyperref}
\usepackage{graphicx}
\usepackage{textcomp}
\usepackage{amsmath,amssymb,bm,bbm,mathrsfs,amscd}
\usepackage{calc}
\usepackage{color}
\usepackage{xcolor}
\usepackage{dsfont}
\usepackage{graphicx}
\usepackage{epstopdf}
\usepackage{epsfig}
\usepackage{tikz}
\usepackage{amsmath}
\usepackage{amsfonts}
\usepackage{amssymb}
\usepackage{rotating}
\usepackage{mathtools}
\usepackage{color}
\usepackage{subfig}
\usepackage{enumerate,pgfplots}
\usepackage{algorithm}
\usepackage{algpseudocode}
\newtheorem{theorem}{Theorem}
\newtheorem{definition}{Definition}
\newtheorem{proposition}{Proposition}
\newtheorem{lemma}{Lemma}
\newtheorem{corollary}{Corollary}
\newtheorem{remark}{Remark}
\newtheorem{assumption}{Assumption}
\newcommand{\ba}{\begin{array}}
\newcommand{\ea}{\end{array}}

\newcommand{\be}{\begin{equation}}
\newcommand{\ee}{\end{equation}}

\newcommand{\ds}{\displaystyle}

\newcommand{\mc}{\mathcal}

\def\1{\boldsymbol{1}}
\def\0{\boldsymbol{0}}

\newcommand{\R}{\mathbb{R}}

\newcommand{\V}{\mathcal{V}}

\newcommand{\G}{\mathcal{G}}

\DeclareMathOperator*{\argmin}{argmin}

\def\R{\mathbb{R}}

\tikzstyle{v_c}=[circle, draw,inner sep=2pt, minimum width=12pt, color=blue]
\tikzstyle{v_a}=[circle, draw,inner sep=2pt, minimum width=12pt, color=red]
\tikzstyle{edge} = [draw,thick,-,font=\small ]
\tikzstyle{label} = [draw,fill=black,font=\normalsize]

\def\G{{\mathcal G}}

\def\BibTeX{{\rm B\kern-.05em{\sc i\kern-.025em b}\kern-.08em
	T\kern-.1667em\lower.7ex\hbox{E}\kern-.125emX}}
\def\MA#1{\textcolor{black}{#1}}

\usepackage{balance}

\begin{document}
\title{\LARGE \bf Predictive Control Strategies for Sustaining Innovation Adoption \\on Multilayer Social Networks}
\author{Martina~Alutto, Qiulin Xu, Fabrizio Dabbene, Hideaki Ishii and Chiara Ravazzi 
	\thanks{Martina~Alutto is with the Department of Decision and Control Systems, KTH Royal Institute of Technology, SE-100 44 Stockholm, Sweden (e-mail: alutto@kth.se). Fabrizio Dabbene and Chiara Ravazzi are with the Institute of Electronics, Computer and Telecommunication Engineering, National Research Council of Italy, Politecnico di Torino, 10129 Torino, Italy (e-mail: {\{chiararavazzi;\,fabriziodabbene;\}@cnr.it}). Qiulin Xu was with the Department of Computer Science, Institute of Science Tokyo, Yokohama 226-8502, Japan (e-mail: xu.q.76b5@m.isct.ac.jp). Hideaki Ishii is with the Department of Information Physics and Computing, The University of Tokyo, Tokyo 113-8654, Japan (e-mail:hideaki\_ishii@ipc.i.u-tokyo.ac.jp).}
	\thanks{This work has been supported by the European Union -- Next Generation EU, Mission 4, Component 1, under the PRIN project {\em{TECHIE: A control and network-based approach for fostering the adoption of new technologies in the ecological transition}}, Cod. 2022KPHA24, CUP Master: D53D23001320006, CUP: B53D23002760006, and by Japan Science and Technology Agency as part of the ASPIRE program under Grant No. JPMJAP2402.}%
}

\thinmuskip=0.1mu  
\medmuskip=0.2mu   
\thickmuskip=0.2mu 

\maketitle

\begin{abstract}
In this paper, we propose a novel compartmental model that couples adoption dynamics and opinion formation over a multilayer network, and we address an optimal control problem within this context. 
\MA{The diffusion of innovation is driven by social imitation through an exposure network and by the evolution of opinions resulting from social interactions. Individuals may temporarily abandon the technology or service due to dissatisfaction before reconsidering.}
The analysis of the equilibrium points and their stability allows us to analytically derive the necessary and sufficient conditions guaranteeing the persistence of adoption.
Building on this result, we introduce a Model Predictive Control strategy that dynamically adapts interventions to the expected evolution of the system. We compare three classes of policies: \MA{influencing opinions, promoting adoption through direct incentives, and reducing dissatisfaction by improving the technology. Numerical results highlight the potential of predictive and adaptive strategies to support the sustainable diffusion of behaviors.}
\end{abstract}

\graphicspath{{Plots}}
\section{Introduction} 
Diffusion of innovations refers to the process through which new ideas, technologies, or behaviors spread within a social system \cite{Rogers2010}. 
Although they occur in different contexts, social diffusion and epidemic processes share a similar underlying structure: both are driven by local interactions on networks, where individual states change through contacts with neighbors or via spontaneous transitions and this can lead to large-scale propagation \cite{daley1964epidemics}. Peer interactions have a significant impact on an individual's state during the dissemination of information \cite{bikhchandani1992theory}, rumors, and new technologies or behaviors \cite{Bass1969}. 

\MA{Classical economic models such as} the Bass model \cite{Bass1969} describe adoption as the result of individual decision-making and peer influence, \MA{while threshold-based and network-based frameworks \cite{Granovetter1978, Kempe2003, BRESCHI2023103651, VILLA2024106106} demonstrate that social diffusion is a form of complex contagion, where the likelihood of adoption varies nonlinearly with population status \cite{centola2018behavior}.}
Epidemic-inspired compartmental models offer a complementary perspective. Originally developed to describe disease dynamics \cite{Kermack.McKendrick:1927}, 
classical formulations such as SIS and SIR have been adapted to model behavioral and technological adoption \cite{Pastor-Satorras2015}. Variants like the Susceptible-Infected-Vigilant (SIV) model or with additional compartments \MA{(see \cite{Xu2024} and references therein)} 
capture mechanisms such as reversible adoption, dissatisfaction and sensitivity to costs.
Adoption, however, is not only influenced by exposure to others but also by beliefs, attitudes, and perceived social norms \cite{Rogers2010, centola2018behavior}.  
Opinion dynamics models capture these elements: the DeGroot model \cite{DeGroot1974} describes opinion convergence through averaging of neighbors’ views, while 
the Friedkin--Johnsen (FJ) model \cite{Friedkin1990} incorporates individual predispositions. 
\MA{A growing body of literature has coupled epidemic-like dynamics with opinion evolution \cite{Xu2024, Paarporn2017, She2022, Xu2026, granell2013dynamical}, showing that transition rates naturally depend on individual beliefs or awareness, and that this interplay can either boost or hinder the spread of innovations. More broadly, multilayer representations have proven effective for capturing distinct facets of social behavior, e.g., discrepancies between expressed and private opinions \cite{ye2019influence}.
From a control perspective, optimal strategies have been proposed for opinion dynamics \cite{sprenger2024control} and epidemic containment \cite{nowzari2015general} separately, but the joint control of coupled adoption-opinion dynamics over multilayer networks remains largely unexplored.}

\MA{This paper provides theoretical and numerical contributions to fill this gap. First, we propose a coupled adoption-opinion model in which a compartmental Susceptible-Adopter-Dissatisfied system interacts with a Friedkin--Johnsen-type opinion dynamics over possibly distinct network layers (Section \ref{sec:model}). Second, we characterize the adoption-free and adoption-diffused equilibria, deriving stability conditions and proving existence and uniqueness of the latter, a novel contribution with respect to~\cite{Xu2024}, and a necessary prerequisite for the well-posedness of the control design (Section~\ref{sec:stability}).
Third, we formulate an optimal control problem with three complementary intervention mechanisms (opinion shaping, adoption propensity enhancement, and dissatisfaction reduction) implemented through a nonlinear Model Predictive Control (MPC) scheme for which we establish recursive feasibility and asymptotic convergence (Section~\ref{sec:mpc}). Fourth, we provide numerical simulations (Section \ref{sec:simulations}) illustrating the qualitative effect of the proposed control strategies. Some of these results appeared in preliminary form in \cite{cdc2025adoption}, while the present version includes new results, complete proofs and additional discussions.}\medskip

{\it Notation:} We denote by $\R$ and $\R_{+}$ the sets of real and nonnegative real numbers. 
The all-1 vector and the all-0 vector are denoted by $\1$ and $\0$ respectively. The identity matrix and the all-0 matrix are denoted by $I$ and $\mathbb{O}$, respectively.
The transpose of a matrix $A$ is denoted by $A^T$. 
For $x$ in $\R^n$, let $||x||_1=\sum_i|x_i|$ and $||x||_{\infty}=\max_i|x_i|$ be its $l_1$- and $l_{\infty}$- norms, respectively, \MA{$\mathrm{diag}(x)$ denotes the diagonal matrix whose diagonal coincides with $x$ and $x_{\max}:=\max_i x_i$, $x_{\min}:=\min_i x_i$}.
For an irreducible matrix $A$ in $\R_+^{n\times n}$, we let $\rho(A)$ denote the spectral radius of $A$. 
Inequalities between two matrices $A$ and $B$ in $\mathbb{R}^{n \times m}$ are entry-wise, i.e., $A \leq B$ means $A_{ij} \leq B_{ij}$ for all $i,j$, while $A < B$ means $A_{ij} < B_{ij}$ for all $i,j$.

\section{Model Description}\label{sec:model}
We introduce a novel adoption-opinion model, which 
describes the diffusion of a service, technology or innovation in a population over a two-layer network. 
We consider a closed population partitioned into communities $\mathcal{V} = \{1, \dots, n\}$, each representing a group of individuals with similar socio-demographic aspects, e.g., age, education, geographic area or mobility level. 
Within each community $i \in \mathcal{V}$, individuals are classified into three compartments:
\begin{itemize}
	\item $s_i(t)$: the fraction of susceptibles (who have not yet adopted the service/technology/innovation at time $t\geq 0$ but may do so in the future through physical interactions with adopters);
	\item $a_i(t)$: the fraction of adopters at time $t\geq 0$;
	\item $d_i(t)$: the fraction of dissatisfied individuals at time $t\geq 0$, including those who previously adopted but are temporarily dissatisfied, or those unconvinced of its benefits.
\end{itemize}

%
Each community holds also an opinion $x_i(t) \in [0,1]$, which represents the average predisposition and perception of the benefits deriving from the adoption of the innovation. 

In this two-layer model, the first layer describes adoption evolution through physical interactions, where adoption likelihood increases with observed adoption in neighboring communities. The second layer describes opinion evolution, driven by social influence and local adoption dynamics. It is worth noting that the two network topologies may differ.


\subsection{Adoption dynamics}
The adoption dynamics build upon the epidemic SIV framework proposed in \cite{Xu2024}. \MA{Adoption spreads through pairwise interactions: susceptible individuals may adopt upon contact with adopters, reflecting the role of physical contagion}. 
After experiencing the adoption, individuals may become dissatisfied and decide to abandon it. Dissatisfaction does not imply a permanent exit from the adoption process, as dissatisfied individuals can regain openness to adoption and return to the susceptible state. Hence, negative experiences may temporarily hinder adoption but do not preclude renewed uptake once opinions improve.
\MA{To capture the interplay between opinions and adoption, we assume that the transition rates depend linearly on the opinion variable $x_i$, which measures the predisposition of community $i$ toward the innovation. Favorable opinions ($x_i$ close to $1$) increase the propensity to adopt and recover from dissatisfaction, while unfavorable opinions increase the tendency to abandon or resist adoption.}

The \textit{physical interaction} network is modeled as a weighted directed graph $\mathcal{G} = (\mathcal{V}, \mathcal{E}, W)$. Here, $\mathcal{V}$ is the set of nodes representing communities, $\mathcal{E}$ is the set of directed links, and $W \in \mathbb{R}_+^{n \times n}$ a nonnegative weight matrix, with $W_{ij} > 0$ if and only if there is a link $(i,j) \in \mathcal{E}$. The neighborhood of node $i$ is defined as $\mathcal{N}_i =\{ j\in \mathcal{V} : \,(i, j) \in \mathcal{E}\}$.
The discrete-time evolution of adoption in community $i$ is described by the following equations:
\begin{align}\label{eq:adoption-model} 
	s_i(t+1) &=s_i(t) - \MA{\beta_ix_i(t)} s_i(t)\displaystyle\sum\nolimits_{j \in \mc N_i} W_{ij} a_j(t) \nonumber\\ 
	& \qquad+ \MA{\gamma_ix_i(t)} d_i(t) -  \MA{\theta_i(1-x_i(t))}s_i(t), \\
	a_i(t+1) &= a_i(t) + \MA{\beta_ix_i(t)}s_i(t)\!\! \displaystyle\sum\nolimits_{j \in \mc N_i} W_{ij} a_j(t) - \delta_i a_i(t), \nonumber \\
	d_i(t+1) &= d_i(t) - \MA{\gamma_ix_i(t)}d_i(t) \!+\! \MA{\theta_i(1-x_i(t))}s_i(t) + \delta_i a_i(t).  \nonumber
\end{align}
\MA{The transition rates $\beta_i, \gamma_i, \theta_i \in (0,1)$ denote the community-specific rates of the susceptibility to adoption, recovery from dissatisfaction and the tendency of susceptibles to become dissatisfied, respectively. 
While $\delta_i \in (0,1]$ is the dismissal rate, capturing post-adoption dissatisfaction due to factors independent of opinion (e.g., service reliability or costs). The framework accommodates more general monotone specifications of the opinion-dependent rates.}

\subsection{Opinion dynamics}
Opinions evolve over time as communities compare their own views with those of others and with the observed adoption levels.
Social interactions are described through a weighted directed graph $\tilde{\mc G}= (\mc V, \tilde{\mc E}, \tilde{W})$, where $\tilde{\mc E} \subseteq \mc V \times \mc V$ is the set of directed links and $\tilde{W} \in \R_+^{n \times n}$ is a nonnegative weight matrix, referred to as the \textit{social interaction} matrix. The social neighborhood of a node $i \in \mc V$ is defined as $\tilde{\mathcal{N}}_i = \{j\in \mc V : ( i,j) \in \tilde{\mc E} \}$.
Building on the FJ framework \cite{Friedkin1990}, we extend the opinion update rule to
\begin{equation}\label{eq:opinion-model} x_i(t+1) =\alpha_i x_i(0) + \lambda_i  \mspace{-3mu} \sum\nolimits_{j \in \tilde{\mc N}_i} \mspace{-5mu} \tilde{W}_{ij}x_j(t) + \xi_i  \mspace{-3mu} \sum\nolimits_{j \in \mc N_i} \mspace{-5mu}W_{ij}a_j(t)\,,
\end{equation}
where $x_i(0)$ is the initial predisposition of community $i$, and the coefficients $\alpha_i,\, \lambda_i,\, \xi_i \,\, \geq 0$ satisfy $\alpha_i + \lambda_i + \xi_i = 1$ for all $i$. These parameters weigh three different sources of influence: intrinsic predisposition ($\alpha_i$), the social neighborhood ($\lambda_i$), and the observed adoption behavior in the physical layer ($\xi_i$). 

\subsection{Coupled adoption-opinion model}
\MA{The proposed coupled model \eqref{eq:adoption-model}--\eqref{eq:opinion-model} builds on two classical models: the discrete-time networked SIV \cite{Xu2024}, reinterpreted for innovation diffusion and the Friedkin--Johnsen model \cite{Friedkin1990}. The main differences are that the adoption rate depends on opinions, and the opinion dynamics includes stubbornness and feedback from observed adoption across neighboring communities.}
The next assumption guarantees the well-posedness.
\begin{assumption}\label{ass:ass1}
\begin{itemize}
    \item[(i)] Both $W$ and $\tilde{W}$ are row-stochastic. $W$ is irreducible. 
	Moreover, for any node $i\in\V$ it holds that \MA{$\lambda_i >0$ and there exists a path in $\widetilde{\G}$ from $i$ to $j$ with $\lambda_j+\xi_j<1$ and $x_j(0)>0$}.
    \item[(ii)] $\gamma_i + \theta_i <1$ for all $i \in\mc V$.
    \item[(iii)] \MA{$\beta_i + \theta_i <1$ for all $i \in\mc V$}.
\end{itemize}
\end{assumption}

Condition (i) ensures that the physical network $\mc G$ is connected and that intrinsic opinions propagate through the social network $\tilde{\mc G}$, preventing communities' beliefs from being determined solely by neighbors. \MA{Conditions (ii)--(iii) ensure that transition rates remain consistent.}

\begin{proposition}\label{prop:invariant}
	Consider the adoption-opinion model \eqref{eq:adoption-model}--\eqref{eq:opinion-model} under Assumption \ref{ass:ass1}. Then, if $s(0), a(0), d(0)$ are in $[0,1]^n$ and $s(0)+ a(0)+ d(0) = \1$, then $s(t), a(t),d(t)$ in $[0,1]^n$ and $s(t)+a(t)+d(t)=\1$ for all $t\geq0$. Moreover, if $x(0)$ in $[0,1]^n$, then $x(t)$ in $[0,1]^n$ for all $t \geq 0$.
\end{proposition}

\MA{The proof is omitted since it follows by induction: Assumption~\ref{ass:ass1}(ii)-(iii) ensures that all state variables remain nonnegative and the simplex constraint is preserved at each step, while the convex combination structure of~\eqref{eq:opinion-model} guarantees $x(t)\in[0,1]^n$ for all $t\geq 0$.}

Due to Proposition \ref{prop:invariant}, for the adoption process we can track just two state variables. From now on, we will consider the following dynamics in vectorial form:  
\begin{align}\label{eq:vector_model} 
	a(t+1) \!\!&=a(t) \!\!+\!\! B \mathrm{diag}(x(t)) \!\!\mathrm{diag}(\1\! -\! a(t)\!-\! d(t)) W a(t)\! -\! \Delta a(t),\nonumber \\
	d(t+1) \!\!&= d(t) - \Gamma \mathrm{diag}(x(t)) d(t)  + \Delta a(t)+\nonumber \\
	&\qquad+ \Theta (I - \mathrm{diag}(x(t))) (\1 -a(t)- d(t)), \\
	x(t+1) \!\!&= (I-\Lambda-\Xi) x(0) + \Lambda \tilde{W} x(t) + \Xi W a(t),\nonumber
\end{align}
where $\Delta :=  \mathrm{diag}(\delta)$, $B :=  \mathrm{diag}(\beta)$, $\Gamma :=  \mathrm{diag}(\gamma)$, $\Theta :=  \mathrm{diag}(\theta)$, $\Lambda:=\mathrm{diag}(\lambda)$, and $\Xi:=\mathrm{diag}(\xi)$ \MA{and initial condition $a(0), d(0), x(0)$ in $[0,1]^n$ with $a(0)+ d(0) \leq \1$}.

%
The next result provides superior and inferior limits for the opinions vector.
\begin{proposition}\label{prop:limsupinf_x}
    Consider the adoption-opinion model \eqref{eq:vector_model} under Assumption \ref{ass:ass1}. Denote 
    \begin{gather}
    x^*:=(I-\Lambda \tilde{W})^{-1}(I-\Lambda-\Xi)x(0), \label{eq:x^*}\\
    x^{\diamond}:=(I-\Lambda \tilde{W})^{-1}\big((I-\Lambda-\Xi)x(0)+\Xi W\1 \big), \label{eq:x-bar}
    \end{gather}
    then $\ds \liminf_{t\to\infty}\ x(t)\geq {x}^*$ and $\ds\limsup_{t\to\infty}\ x(t)\leq x^\diamond$.
\end{proposition}
\begin{proof}
    From Proposition \ref{prop:invariant}, $a(t),d(t)$ belong to $[0,1]^n$ for all $t\geq 0$.  
    Moreover, the opinion dynamics is an affine system with input given by $\0\leq \Xi W a(t)\leq  \Xi W\1$ for all $t\ge 0$. 
    We thus have
    \begin{equation}\label{eq:lowerdyn}
    x(t+1) \geq \Lambda \tilde{W} x(t) + (I-\Lambda-\Xi)x(0),
    \end{equation}
    \begin{equation}\label{eq:upperdyn}
    x(t+1) \leq \Lambda \tilde{W} x(t) + (I-\Lambda-\Xi)x(0)+\Xi W \1.
    \end{equation}
    \MA{Since $\Lambda\tilde{W}$ is a nonnegative matrix with $\rho(\Lambda\tilde{W}) < 1$ (see, e.g., Lemma~5 in~\cite{FRASCA2013212}), the matrix $(I-\Lambda\tilde{W})$ is invertible, and taking the limit yields the result. }
\end{proof}
\begin{remark}
The vectors $x^*$ and $ x^{\diamond}$ correspond to equilibrium opinions when the adoption feedback $\Xi W a(t)$ takes its minimum and maximum level, respectively.
\end{remark}
\begin{corollary}\label{cor:bounds_x}
\MA{Under Assumption~\ref{ass:ass1}, the following hold:
\begin{itemize}
    \item[(i)] $x^*_i > 0$ for every $i \in \mathcal{V}$;
    \item[(ii)] for all $t\geq 0$,
    \be \label{eq:lim-x} x(t)\leq \bar x_{\varepsilon}:= x^{\diamond} + \varepsilon_t, \quad x(t)\geq \underline{x}_{\varepsilon}:= x^{*} -\varepsilon_t , \ee
    where $\varepsilon_t :=(\Lambda \tilde{W})^{t} \max\{|x(0) - x^{\diamond}|, |x(0) - x^*|\}$ converges to $\mathbf{0}$ exponentially.
\end{itemize}}
\end{corollary}
\begin{proof}
\MA{(i) By Assumption~\ref{ass:ass1}(i), the $j$-th component of $(I-\Lambda-\Xi)x(0)$ is strictly positive. Since $\lambda_i > 0$ for all $i$, the entries of $(\Lambda\tilde{W})^k$ along this path are strictly positive, hence the $(i,j)$-entry of $(I-\Lambda\tilde{W})^{-1}$ is positive, yielding $x_i^* > 0$.\\
(ii) Iterating \eqref{eq:lowerdyn}--\eqref{eq:upperdyn} and $\rho(\Lambda\tilde{W}) < 1$ yields the result.}
\end{proof}

\section{Equilibrium and stability analysis}\label{sec:stability}
In this section, we study the equilibria of the system in \eqref{eq:vector_model}, analyze their stability, focusing on two equilibrium configurations: (i) the adoption-free equilibrium (AFE), where no individuals adopt the innovation, and (ii) the adoption-diffused equilibrium (ADE), where a positive fraction of adopters persists across communities. Characterizing these two equilibria allows us to derive conditions under which adoption either fails to spread or becomes self-sustaining.

\subsection{Adoption-free equilibrium}
The following proposition, whose proof is straightforward and omitted for brevity, guarantees that there exists an adoption-free equilibrium. 
To this end, define~$\Psi(x) := \big( (\Gamma- \Theta)\mathrm{diag}(x) + \Theta\big)^{-1}\Theta(I-\mathrm{diag}\big(x\big)).$
\begin{proposition}\label{prop:equilibria}
	Given Assumption \ref{ass:ass1}, 
	$\big(\0,\mspace{5mu} d^*, \mspace{5mu} x^* \big)$ is an equilibrium for \eqref{eq:vector_model}
	with $d^* := \Psi(x^*)\1$ and $x^*$ defined in \eqref{eq:x^*}.
\end{proposition}

The following result characterizes the local stability properties of the adoption-free equilibrium. First, define 
\begin{equation}\label{eq:M}M(x) := I - \Delta +B \mathrm{diag}(x)\left(I - \Psi(x)\!\!\right)\!W\,.\end{equation}
\begin{proposition}\label{theo:free-equilibrium}
	Consider the adoption-opinion model \eqref{eq:vector_model} under Assumption \ref{ass:ass1}. Then, the following hold:
	\begin{itemize}
			\item[(i)] If $\rho(M(x^*)) <1$, then the AFE is locally stable.
			\item[(ii)] If $\rho(M(x^*)) >1$, then the AFE is unstable.
	\end{itemize} 
\end{proposition}
\begin{proof}
	Denoting the function $\Phi(\cdot) :=  \Gamma\mathrm{diag}(\cdot) + \Theta(I-\mathrm{diag}(\cdot))$, the Jacobian matrix of \eqref{eq:vector_model} at the AFE will be 
	\begin{equation*}\label{J_1}
		\mspace{-3mu}	J :\!=\!\! \begin{bmatrix}
				\ds I- \Delta +B\mathrm{diag}(x^*)\mathrm{diag}(\1 \!-\! d^*)W  & \ds \mathbb{O} &\ds  \mathbb{O} \\ 
				\Delta -  \Theta (I-\mathrm{diag}(x^*)) & \ds I - \Phi(x^*) &- \Phi(d^*) \!\!\! \\
				\Xi W &\mathbb{O} &  \Lambda \tilde{W}
			\end{bmatrix}.
	\end{equation*}
	The eigenvalues of $J$ are obtained as the union of the eigenvalues of its diagonal blocks. The matrix $I \!-\! \Phi(x^*) = I \!-\! \big( \Gamma\mathrm{diag}(x^*) + \Theta(I-\mathrm{diag}(x^*))\big)$ is diagonal. \MA{Since $\gamma_i x_i^* + \theta_i(1-x_i^*) \geq \min(\gamma_i, \theta_i) > 0$ for any $x_i^* \in [0,1]$, Assumption~\ref{ass:ass1}(ii) ensures $\rho(I - \Phi(x^*)) < 1$.} The matrix $\Lambda \tilde{W}$ is strictly sub-stochastic, and its largest eigenvalue is less than $1$. Therefore, the local stability of the AFE depends only on the real part of the eigenvalues of the matrix $M(x^*) = I- \Delta +B\mathrm{diag}(x^*)\mathrm{diag}(\1 - d^*)W$.
	By the linearization theorem, the AFE is locally stable if $\rho(M(x^*)) < 1$, and unstable if $\rho(M(x^*)) > 1$.
\end{proof}\smallskip


Building on the previous results, we can now provide a condition for the global stability for the adoption-free equilibrium. The following theorem is one of the main results of this paper.
\begin{theorem}\label{theo:free-equilibrium2}
Consider the adoption-opinion model \eqref{eq:vector_model} under Assumption \ref{ass:ass1}. If $\rho(M({x}^{\diamond}))<1$, then the AFE is globally asymptotically stable.
\end{theorem}
\begin{proof}
\MA{We define $y(t):=a(t)+d(t)$ and sum the first two equations in \eqref{eq:vector_model}. Using Corollary \ref{cor:bounds_x} and \eqref{eq:lim-x} we bound
\begin{align*}
y(t+1)=&y(t)+B\mathrm{diag}(x(t))\mathrm{diag}(\1-y(t))Wa(t)+\\
&-\Gamma \mathrm{diag}(x(t))d(t)+\Theta(I-\mathrm{diag}(x(t)))(\1-y(t))\\
\geq& y(t)-\Gamma \mathrm{diag}(\overline{x}_{\epsilon})y(t)+\Theta(I-\mathrm{diag}(\overline{x}_{\epsilon}))(\1-y(t)).
\end{align*}
Since $\varepsilon_t \to\, \0$ exponentially, taking the inferior limit yields
\begin{align*}
	\0\geq &-[\Gamma \mathrm{diag}({x}^{\diamond})+\Theta(I-\mathrm{diag}({x}^{\diamond}))] \liminf_{t\rightarrow\infty} y(t)+\\
    &\Theta(I-\mathrm{diag}({x}^{\diamond}))\1.
\end{align*}
Rearranging the terms, we obtain \be \label{eq:lemma}\ds \liminf_{t\rightarrow\infty} y(t)\geq \Psi({x}^{\diamond})\1.\ee}
Using \eqref{eq:M}, we can rewrite the adopter dynamics by adding and subtracting the term $\Psi(\overline{x}_{\epsilon})$ as
	\MA{$$a(t+1)\leq M(\overline{x}_{\epsilon})a(t)+B\mathrm{diag}(\overline{x}_{\epsilon})\mathrm{diag}(\Psi(\overline{x}_{\epsilon})\1-a(t)-d(t))a(t),$$}
	for all $t \geq 0$.
	From \eqref{eq:lemma}, for any $\epsilon > 0$ there exists $T > 0$ such that
	$a(t)+d(t)>\Psi(x^{\diamond})\1-\epsilon \MA{\1}$ for all $t \ge T$.
    \MA{Since $\Psi$ is componentwise decreasing and $\overline{x}_{\epsilon} = x^\diamond + \varepsilon_t \geq x^\diamond$, we have $\Psi(\overline{x}_{\epsilon}) \leq \Psi(x^\diamond)$, so $\Psi(\overline{x}_{\epsilon})\1-a(t)-d(t) \leq \epsilon\1$.}
	Since $\rho(M(x^{\diamond})) < 1$ and $\varepsilon_t$ goes to $0$, choosing $T$ large enough and $\epsilon$ small enough ensures
	\begin{equation}\label{eq:eq1}a(t+1)\leq (M(x^{\diamond})+\epsilon I)a(t),  \quad \forall\, t \ge T.\end{equation}
	By the Perron–Frobenius theorem \cite{Horn1991}, there exist a nonnegative vector $v\neq 0$ and a constant \MA{$\omega\in(0,1)$ such that $v^\top (M(x^{\diamond})+\epsilon I) \le \omega\, v^\top$}. Multiplying both sides of \eqref{eq:eq1} by $v^\top$, we obtain 
	$$v^\top a(t+1)\ \le\ v^\top  (M(x^{\diamond})+\epsilon I) a(t)\ \le\ \MA{\omega}\, v^\top a(t), \quad \forall t \geq T,$$
	which implies that $v^\top a(t)$ goes to $0$ exponentially fast and then $a(t)$ converges to $\0$ for any initial state $a(0)$ in $[0,1]^n$.
Since $\Lambda\tilde{W}$ is Schur stable under Assumption~\ref{ass:ass1}(i) (see \cite{FRASCA2013212}), $x(t)$ asymptotically converges to $x^*$, that is the unique stable equilibrium of the FJ model.

It remains to show that $d(t)$ asymptotically converges to $d^*$. 
For the sake of conciseness, we adopt the shorthand:
$\theta_i(t) = \theta_i (1-x_i(t))$, $\gamma_i(t)= \gamma_i x_i(t)$.
\MA{We define the error 
$$e_i(t) :=d_i(t) - \frac{\delta_i a_i(t) + \theta_i (t)(1-a_i(t))}{\theta_i (t) + \gamma_i(t)}\,,$$
for all $i \in \V$. Then, we get
$ e (t+ 1) = [I- \Gamma \mathrm{diag}(x(t)) - \Theta(I - \mathrm{diag}(x(t))) ]e (t) + b(t)\,$}
with 
\begin{align*}
	b_i (t)& :=\! \left[ \frac{\delta_i \!-\! \theta_i (t)}{\theta_i \!(t) \!+\! \gamma_i(t)} \!+\! \frac{\theta_i (t+1)\!-\! \delta_i) (1 \!-\! \delta_i)}{\theta_i (t+1) \!+\! \gamma_i(t+ 1)} \right]\!\! a_i(t) \\
	&\mspace{-5mu}+ \frac{(\theta_i (t+1) \!-\! \delta_i) (1 \!-\! a_i(t) - d_i(t))}{\theta_i (t+1) \!+\! \gamma_i(t+ 1)}\beta_i x_i(t)\!\! \sum\nolimits_{j \in \mc N_i} \! W_{ij} a_j(t) \\
	&+ \frac{\theta_i(t)}{\theta_i(t) + \gamma_i(t)} - \frac{\theta_i (t+1)}{\theta_i (t+1) + \gamma_i(t+ 1)}\,.
\end{align*}
Define $\eta :=  1-\min_i (\gamma_ix_i+\theta_i(1-x_i))$. Then, from Assumption \ref{ass:ass1} we have that \MA{$\|I- \Gamma \mathrm{diag}(x(t)) - \Theta(I - \mathrm{diag}(x(t)))\|_\infty \leq \eta <1$}. Since $a(t)$ asymptotically converges to $\0$ and $x(t)$ is also convergent to $x^*$, it follows that $\|b(t)\|_\infty$ vanishes as $t \to \infty$.
We thus have
$$\|e(t+1) \|_\infty \leq \eta^{t+1}\|e(0)\|_\infty+\sum_{s=0}^t\eta^{t-s-1}\|b(s)\|_\infty,$$
which converges to zero (the second term is a discrete convolution with an exponentially decaying sequence $\|b(s)\|_\infty$ tending to zero). This completes the proof. 
\end{proof}

\begin{remark}
    \MA{Analogously to the reproduction number in epidemic modeling, we define the opinion-dependent reproduction number $R_0^A(t) := \rho(M(x(t))$ for the adoption-opinion model \eqref{eq:vector_model}.} The spectral radius $\rho(M(x))$ is increasing in $x$: favorable opinions raise the threshold for adoption takeoff, while unfavorable ones reduce it. Given an upper bound on opinions, the corresponding maximum reproduction number provides a conservative threshold for determining whether adoption can take off. Conversely, a minimum reproduction number can be defined by considering the worst-case scenario, in which all communities have the lowest possible opinions.
\end{remark}

The following proposition provides more interpretable sufficient conditions for stability analysis.
\begin{proposition}\label{prop:cond-suff-afe}
	Consider the adoption-opinion model~\eqref{eq:vector_model} under Assumption~\ref{ass:ass1}.
	Then, the following hold:
	\begin{enumerate}
		\item[(i)] If for all $i\in\mathcal{V}$,
		\begin{equation}\label{eq:cond-x^*} \beta_i\,\frac{\gamma_i\, x_i^*}{\gamma_i x_i^* + \theta_i(1-x_i^*)}
		\;>\;\delta_i,
		\end{equation}
		then the AFE is unstable.
		\item[(ii)] If for all $i\in\mathcal{V}$,
		\begin{equation}\label{eq:cond-barx}
        \beta_i \frac{\gamma_ix^\diamond_i}{\gamma_i x^\diamond_i + \theta_i(1-x^\diamond_i)}< \delta_i,
        \end{equation}		
        then the AFE is globally asymptotically stable.
	\end{enumerate}
\end{proposition}
\begin{proof}
	(i) 
    By the Collatz--Wielandt formula for the spectral radius \cite{Collatz1942} and \eqref{eq:M}, for any $v>\0$ we have
	$$\min_i \frac{(M(x)v)_i}{v_i} \;\le\; \rho(M(x)) \;\le\; \max_i \frac{(M(x)v)_i}{v_i}.$$
	Choosing $v = \1$, we get the following bounds:
	\begin{equation}\label{eq:collatz} \min_i \sum\nolimits_{j\in\mc N_i} M_{ij}(x) \;\le\; \rho(M(x)) \;\le\; \max_i \sum\nolimits_{j\in\mc N_i} M_{ij}(x).\end{equation}
	Evaluating the left-hand side of \eqref{eq:collatz} at $x = x^*$, we have
 	$$\sum\nolimits_{j\in\mc N_i} M_{ij}(x^*) = 1 - \delta_i + \beta_i \frac{\gamma_i x_i^*}{\gamma_i x_i^* + \theta_i(1-x_i^*)}>1, \quad \forall i,$$
 	which follows from \eqref{eq:cond-x^*}. Thus $\rho(M(x^*)) > 1$ and from Proposition \ref{theo:free-equilibrium}(ii), the AFE is unstable.\\
	(ii) Evaluating the right-hand side of \eqref{eq:collatz} at $x = x^\diamond$, we have
	$$\sum\nolimits_{j\in\mc N_i} M_{ij}(x^\diamond) = 1 - \delta_i + \beta_i \frac{\gamma_ix^\diamond_i}{\gamma_i x^\diamond_i + \theta_i(1-x^\diamond_i)}<1, \quad \forall i,$$
	which follows from \eqref{eq:cond-barx}. Thus \MA{$\rho(M(x^\diamond)) < 1$} and by Theorem~\ref{theo:free-equilibrium2}, the AFE is globally asymptotically stable.
\end{proof}\medskip
\begin{remark}
	Conditions \eqref{eq:cond-x^*} and \eqref{eq:cond-barx} highlight the interplay between the adoption rate $\beta_i$, opinions $x_i$, and the dissatisfaction rate $\delta_i$ in determining the stability of the AFE. Specifically, in \eqref{eq:cond-x^*}, if $\beta_i$ is sufficiently large at the equilibrium $x_i^*$, the effective adoption rate exceeds $\delta_i$, rendering the equilibrium unstable. Conversely, in \eqref{eq:cond-barx}, if $\beta_i$ is too small relative to $\delta_i$ even at the maximal opinion \MA{$x^\diamond_i$}, adoption cannot propagate, yielding a conservative sufficient condition for the global stability of the AFE.
\end{remark}

\subsection{Adoption-diffused equilibrium}
\MA{Having established stability of the adoption-free equilibrium, we now investigate the existence, uniqueness, and stability of an adoption-diffused equilibrium, i.e., a steady state with strictly positive adoption across the network. }

\MA{Before stating the next result, we define $\phi_i(x_i) := \theta_i(1-x_i)+\gamma_i x_i$ and $m_i := \min_{x_i \in [x^*_i, x^\diamond_i]} \phi_i(x_i)$.}

\begin{theorem}\label{theo:diffused-equilibrium}
	Consider the adoption-opinion model \eqref{eq:vector_model} under Assumption \ref{ass:ass1}. If $\rho(M(x^*)) >1$, then the following hold:
	\begin{enumerate}
		\item[(i)] There exists an ADE $(a^\dag, d^\dag, x^\dag)$ with $a^\dag\neq \0$. 
		\item[(ii)] \MA{If also it holds that 
        \be \label{eq:hyp1} \max_i \frac{\beta_i}{\delta_i}\left[1 + \frac{\gamma_i+\delta_i}{m_i} + \frac{\gamma_i(1+\theta_i)}{m_i^2}\frac{\xi_{\max}}{1-\lambda_{\max}}\right] \frac{\delta_{\min}}{\delta_{\max}}\leq1,\ee
        then the adoption-diffused equilibrium $(a^\dag, d^\dag, x^\dag)$ is unique and globally asymptotically stable. }
	\end{enumerate}
\end{theorem}
\begin{proof}
    (i) \MA{The proof follows arguments in \cite{Xu2024}, adapted to the present framework}.
    By hypothesis, let $q := \rho(M(x^*)) - 1 > 0$. Since the map $M \mapsto \rho(M)$ is continuous with respect to $M$ \cite{Horn1991}, there exists \MA{$\sigma > 0$ such that
    $$\text{ if } \|x - x^*\|_\infty < \sigma \text{ then }  \rho(M(x))>1+\frac q2>1.$$
	From Proposition \ref{prop:limsupinf_x}, there exists a time $T > 0$ such that $\|x-x^*\|_\infty<\sigma$ for all $t \geq T$. In particular, setting $\underline{x}:= x^* - \epsilon\boldsymbol{1}$ for any $\epsilon\in (0,\sigma)$, we have $x(t) \geq \underline{x}$ for all $t \ge T$ and $\rho(M(\underline{x}))>1$.}
    Since $M(\underline{x})$ is Metzler and irreducible, by the Perron–Frobenius theorem \cite{Varga2009}, there exist $\lambda >0$ and $\mu > \0$ such that 
    \begin{equation}\label{eq:eig} (- \Delta + B \mathrm{diag}(\underline{x})\big(I - \Psi(\underline{x})\big)W)\mu = \lambda \mu,\end{equation}

Consider the following system:
 \begin{align}
		\displaystyle a_i(t+1) =&\! a_i(t) + \beta_i x_i(t)(1-a_i(t)-d_i(t)) \!\!\displaystyle \sum_{j \in \mc N_i} \! W_{ij} a_j(t)-\! \delta_i a_i(t), \nonumber \\
		\displaystyle d_i(t) &=  \frac{\delta_i a_i(t) + \theta_i (1-x_i(t))(1-a_i(t))}{\theta_i (1-x_i(t)) + \gamma_i x_i(t)} , \label{eq:model2} \\ 
		\displaystyle x_i(t+1) =&(1-\lambda_i-\xi_i)x_i(0) + \lambda_i   \sum_{j \in \tilde{\mc N}_i}\tilde{W}_{ij}x_j(t) +  \xi_i \sum_{j \in \mc N_i}W_{ij}a_j(t)\,. \nonumber
\end{align}
Considering the equilibrium of $d_i(t)$, it follows that every adoption-diffused equilibrium of the current system is also an adoption-diffused equilibrium of system \eqref{eq:vector_model}.
Define $\Omega:= [0,1]^{3 n} \setminus \{ (\0, d,x): d,x \in [0,1]^n \}$.
Now let us consider the convex and compact subset of $\Omega$ given by
\begin{equation}\label{eq:invariant_set} \Omega_\varepsilon := \{ (a, d, x) \mid\, \, a_i \in [\varepsilon \mu_i, 1],\,\, d_i ,\,\, x_i \in [0, 1],\,\, \forall i \in \mc V\},\end{equation}
where $\varepsilon \in (0,1)$ is a constant. For any $(a(t), d (t), x(t)) \in \Omega$, there exists a sufficiently small $\varepsilon_1 > 0$, satisfying $a_i(t) = \varepsilon_1 \mu_i$, and $a_j(t) \geq \varepsilon_1 \mu_j$ for all $j \neq i$, so $(a(t), d(t), x(t)) \in \Omega_{\varepsilon_1}$. From \eqref{eq:model2} and \eqref{eq:eig}, it follows that
{\begin{align*}
		a_i ( t + 1)  &
        \geq \varepsilon_1(1+\lambda)\mu_i-\varepsilon_1^2 \mu_i \frac{\gamma_i x_i(t) + \delta_i}{\phi_i(x_i(t))} \beta_i x_i(t)\displaystyle \sum_{j \in \mc N_i} \! W_{ij}  \mu_j.
\end{align*}}
Note that there must exist a sufficiently small $\varepsilon_2 > 0$ such that
$a_i(t+1) \geq \varepsilon_2 \mu_i$ for all $i$.
Hence, $\Omega_{\bar{\varepsilon}}$ is a positive invariant set for system \eqref{eq:model2} for all $\bar{\varepsilon} \in (0, \min\{\varepsilon_1, \varepsilon_2\})$. Therefore, by Brouwer’s fixed-point theorem \cite{khamsi2011introduction}, the system \eqref{eq:model2} and hence the original system \eqref{eq:vector_model} have at least one equilibrium in $\Omega_{\bar{\varepsilon}}$.

\MA{(ii) We now prove uniqueness and global asymptotic stability. For notational convenience, we drop the equilibrium superscript ${}^\dag$ throughout this part of the proof. From the second and third equations in \eqref{eq:model2} at equilibrium, we obtain $d$ and $x$ as functions of $a$:
\begin{align}
d_i(a)&=\frac{\delta_ia_i+\theta_i(1-x_i(a))(1-a_i)}{\theta_i(1-x_i(a))+\gamma_i x_i(a)}, \quad \forall i \label{eq:fixed_points_equation-d} \\
x(a)&=(I- \Lambda\tilde{W})^{-1} \big((I-\Lambda-\Xi) x(0) + \Xi W a\big).
\label{eq:fixed_points_equation-x}
\end{align}
Substituting \eqref{eq:fixed_points_equation-d}--\eqref{eq:fixed_points_equation-x}
into the first equilibrium equation defines the operator $F: [\bar \varepsilon \mu_i, 1]^n \to [\bar \varepsilon \mu_i, 1]^n$ with components
\begin{equation}\label{eq:Ti}
F_i(a) := \frac{\beta_i}{\delta_i} \, x_i(a) \,\left(1-a_i-d_i(a)\right) \sum\nolimits_{\ell \in \mc N_i} W_{i \ell} a_\ell .
\end{equation}
We show that $F$ is a contraction on $([\bar\varepsilon\mu,1]^n,\|\cdot\|_\infty)$ by bounding the infinity norm of its Jacobian $DF(a)$.
Given $f_i(a) := x_i(a) \left(1-a_i-d_i(a)\right)$, differentiating gives
\begin{equation}\label{eq:DT}
\frac{\partial F_i}{\partial a_j} = \frac{\beta_i}{\delta_i}\Bigg[\frac{\partial f_i}{\partial a_j}\sum\nolimits_{\ell \in \mc N_i} W_{i \ell} a_\ell  + f_i(a)\,W_{ij}\Bigg] \quad \forall i,j.
\end{equation}
Applying the triangle inequality, and using $W_{ij}\ge 0$, $a_j>0$ for all $j$, $f_i(a)\ge 0$ for all $i$,
\begin{equation}\label{eq:row-sum}
\sum_{j=1}^n \left|\frac{\partial F_i}{\partial a_j}\right|
\le \frac{\beta_i}{\delta_i} \Bigg[\sum_{\ell \in \mc N_i} W_{i \ell} a_\ell\sum_{j=1}^n\left|\frac{\partial f_i}{\partial a_j}\right| + f_i(a)\Bigg],
\end{equation}
where we used row-stochasticity of $W$.
Note that 
$$\frac{\partial f_i}{\partial a_j} = \frac{\partial x_i}{\partial a_j}(1-a_i-d_i(a)) + x_i(a)\Big(-\frac{\partial d_i}{\partial x_i}\Big)\frac{\partial x_i}{\partial a_j}, \quad \forall j\neq i.$$
Including the term for $j=i$ and omitting the algebraic details for brevity, we obtain
\begin{equation*}
\frac{\partial f_i}{\partial a_j} = \underbrace{\frac{\gamma_i x_i(1-a_i)(\phi_i(x_i)+\theta_i)-\delta_i\theta_i a_i}{\phi^2_i(x_i)}}_{\displaystyle=:A_i(a)} \frac{\partial x_i}{\partial a_j}
- \mathds{1}_{j=i}\, x_i \frac{\gamma_i x_i+\delta_i}{\phi_i(x_i)},
\end{equation*}
where $ \mathds{1}_{j=i}$ is the indicator function. 
Since $(a,d,x)\in\Omega_{\bar\varepsilon}$ implies $1-a_i-d_i\geq0$, one can verify that
$$1-a_i-d_i(a) = \frac{\gamma_i x_i(1-a_i)-\delta_i a_i}{\phi_i(x_i)} \geq 0$$
implies $A_i(a)\ge 0$. In addition, since $\frac{\partial x_i}{\partial a_j}\ge 0$ for all $j$ (from \eqref{eq:fixed_points_equation-x}), we conclude that $\frac{\partial f_i}{\partial a_j}\ge 0$ for all $j\neq i$.
Summing over $j$ and taking absolute values,
\begin{equation}\label{eq:sum-df}
\sum_{j=1}^n \left|\frac{\partial f_i}{\partial a_j}\right|
\le A_i(a)\sum_{j=1}^n\frac{\partial x_i}{\partial a_j} + x_i\,\frac{\gamma_i x_i+\delta_i}{\phi_i(x_i)}.
\end{equation}
From \eqref{eq:fixed_points_equation-x}, $D_a x = (I-\Lambda\widetilde{W})^{-1}\Xi W$. Since $\widetilde{W}$ and $W$ are row-stochastic and $\lambda_{\max}<1$, the Neumann series gives 
$$\sum_{j=i}^n \frac{\partial x_i}{\partial a_j} \le \frac{\xi_{\max}}{1-\lambda_{\max}}.$$ 
Since $x_i(1-a_i)\le 1$, the term $-\delta_i\theta_i a_i< 0$, and $\phi_i(x_i)+\theta_i \le 1+\theta_i$, we get $A_i(a) < \gamma_i(1+\theta_i)/ m_i^2$.
Since $x_i\le 1$ and $\phi_i(x_i)\ge m_i$, we get from \eqref{eq:sum-df} that
\begin{equation}
\sum_{j=1}^n\left|\frac{\partial f_i}{\partial a_j}\right|
< \frac{\gamma_i(1+\theta_i)}{m_i^2}\frac{\xi_{\max}}{1-\lambda_{\max}} + \frac{\gamma_i+\delta_i}{m_i}.
\end{equation}
Substituting this into \eqref{eq:row-sum} and using $f_i(a)\le 1$,
\begin{equation}
\sum_{j=1}^n\left|\frac{\partial F_i}{\partial a_j}\right|
< \frac{\beta_i}{\delta_i}\left[1 + \frac{\gamma_i+\delta_i}{m_i} + \frac{\gamma_i(1+\theta_i)}{m_i^2}\frac{\xi_{\max}}{1-\lambda_{\max}}\right]=: C_i.
\end{equation}}
\MA{Taking the maximum over $i$, under the assumption \eqref{eq:hyp1},
we have $\|DF(a)\|_\infty < \max_i C_i < 1$ uniformly in $a\in[\bar\varepsilon\mu,1]^n$. This implies $F$ is a strict contraction on the complete metric space $([\bar\varepsilon\mu,1]^n,\|\cdot\|_\infty)$. By the Banach fixed-point theorem \cite{khamsi2011introduction}, $T$ admits a unique fixed point that is $a^\dag\in[\bar\varepsilon\mu,1]^n$ and for every $a^0\in[\bar\varepsilon\mu,1]^n$ the iterates $a^{k+1}=F(a^k)$ satisfy $\|a^k - a^\dag\|_\infty \le \max_i C_i \|a^0-a^\dag\|_\infty \to 0$ as $k\to\infty$.}

\MA{We now consider the full dynamics and its asymptotic behavior as $t\rightarrow\infty$. From \eqref{eq:vector_model}, the $a$-subsystem can be rewritten as $$a_i(t+1)=(1-\delta_i)a_i(t)+\delta_i F_i(a(t)).$$ Define $G(a)=(I-\Delta)a+\Delta F(a)$, then $a(t+1)=G(a(t))$. Then,
$$\|G(a)-G(b)\|_\infty\le (1-\delta_{\min})\|a-b\|_\infty+\delta_{\max} \max_i C_i \|a-b\|_\infty.$$ 
Since $1-\delta_{\min}+\delta_{\max} \max_i C_i<1$ from \eqref{eq:hyp1}, $G$ is a strict contraction, so $a(t)\to a^\dag$ globally and exponentially.
Next, for the $x$-dynamics in \eqref{eq:vector_model}, since $\|\Lambda\tilde W\|_\infty<1$, the linear system is stable and driven by $a(t)$, hence $x(t)\to x^\dag$ with $x^\dag=(I-\Lambda\tilde W)^{-1}\big((I-\Lambda-\Xi)x(0)+\Xi Wa^\dag\big).$
Finally, $d(t)$ is a continuous function of $(a(t),x(t))$ with uniformly positive denominator on $\Omega_{\bar\varepsilon}$, hence $d(t)\to d^\dag$. Therefore $(a(t),d(t),x(t))\to(a^\dag,d^\dag,x^\dag),$ proving global asymptotic stability and uniqueness of the equilibrium of the full system.}
\end{proof}
\begin{remark}
    \MA{It is worth noting that condition~\eqref{eq:hyp1} is sufficient but not necessary for uniqueness and global asymptotic stability. When $\rho(M(x^*))>1$ but~\eqref{eq:hyp1} fails, existence of the ADE is still guaranteed by Theorem~\ref{theo:diffused-equilibrium}(i), while uniqueness and global stability are not claimed and remain open.}
\end{remark}

\section{Formulation of the Control Problems}\label{sec:mpc}
The adoption of new technologies, such as electric vehicles, typically requires external support to overcome skepticism, reduce adoption barriers, and sustain users' engagement over time. 
To model these actions, we introduce a control function \MA{$u_i: [0,+\infty) \to [0,1]$}, representing the intensity of interventions applied to community $i$. The objective is to promote widespread and persistent adoption while accounting for limited resources.
We consider three intervention mechanisms corresponding to different policy levers, studied separately to isolate their effects and compare them. 

\begin{itemize}
    \item[(P1)] \textit{Opinion shaping.} Public information campaigns, educational programs, or targeted communication can shift perceptions in favor of innovation \cite{filippini2021nudging, li2020analysis}. 
    \MA{Inspired by \cite{ravazzi2026optimal}, we model this by modifying the stubbornness term in the opinion update rule:
    \begin{equation}\label{eq:opinion-control} x(t+1) \!=\! (I - \Lambda - \Xi)(x(0) + u^o(t))\! + \Lambda \tilde{W} x(t) + \Xi W a(t),\end{equation}
    where $u^o(t) \in  U^o := [\boldsymbol{0}, \boldsymbol{1}-x(0)]$.} 
    \item[(P2)] \textit{Adoption propensity enhancement.} Even when opinions are favorable, adoption may remain limited by practical barriers (e.g., high costs, complex usage, or lack of compatibility). Subsidies, financial incentives, and technical support can help overcome these obstacles, increasing individuals’ willingness and ability to adopt \cite{silvia2016assessing, li2022diffusion}.
    We incorporate this by modulating the adoption rate:
   \MA{ \begin{equation}\label{eq:beta-control}\beta^a_i(x_i(t)) := \beta_i x_i(t) (1 + u^a_i(t)),\end{equation}
    with $u^a(t) \in  U^a=\textstyle\prod_{i \in \mathcal{V}} \Big[0,\frac{1-\theta_i-\beta_i}{\beta_i}\Big]$, to ensure consistency with Assumption~\ref{ass:ass1}(iii).}
    \item[(P3)] \textit{Dissatisfaction reduction.} Long-term diffusion requires also minimizing dissatisfaction: negative experiences may lead users to abandon the innovation, slowing its spread \cite{rezvani2015advances}. Improving product quality, providing reliable customer service, or enhancing user experience are all measures that reduce dissatisfaction. 
    We model this mechanism by reducing the abandonment rate:
    \MA{\begin{equation}\label{eq:delta-control}\delta^d_i := \delta_i (1 - u^d_i(t)),\end{equation}
    with $u^d(t) \in U^d=[0,1]^n$.}
\end{itemize}
In real applications, control efforts are constrained by limited resources, then $\1^\top u^k(t) \leq C$ for all $t \geq0$ and $k\in\{o,a,d\}$,
where $C > 0$ is the available budget for interventions at each time. 

\subsection{Constant Control Policy}
\MA{We first study constant implementations of the three control mechanisms introduced above.} 
The goal is to simultaneously enhance adoption, reduce dissatisfaction, and limit the cost of interventions. 
For each mechanism $k\in\{o,a,d\}$, we solve
\begin{equation}\label{eq:opt-constant}
\MA{\begin{aligned}
    \bar{u}^k := \argmin_{u^k} \; J(u^k) &= \sum_{i=1}^n \left[
        -Q^a_i (a^{\dagger,k}_{i})^2 + Q^d_i (d^{\dagger,k}_{i})^2 
        + L_i (u^k_i)^2
    \right] \\
    \text{subject to} \quad 
    & \mathbf{1}^\top u^k \leq C, \quad u^k \in \mathcal{U}^k, \\
    & \rho(M^k(x^{*,k})) > 1, \, \text{ and } \eqref{eq:hyp1} \text{ holds,}
\end{aligned}}
\end{equation}
where $a^{\dagger,k}$ and $d^{\dagger,k}$ denote the equilibrium fractions of adopters and dissatisfied individuals under mechanism $k$, and $Q^a_i$, $Q^d_i$, $L_i$ weight the relative priority of promoting adoption, reducing dissatisfaction, and minimizing cost. \MA{The matrix $M^k$ and the opinion lower bound $x^{*,k}$ depend on the selected mechanism. Under (P1), $x^{*,o} = (I-\Lambda\tilde{W})^{-1}(I-\Lambda-\Xi)(x(0)+u^o)$, and $M^o =M(x^{*,o})$ as in~\eqref{eq:M}. Under (P2), $\beta_i$ is replaced by \eqref{eq:beta-control} in~\eqref{eq:M}, yielding $M^a$, while $x^{*,a} = x^*$. Under (P3), $\delta_i$ is replaced by \eqref{eq:delta-control} in~\eqref{eq:M}, giving $M^d$, while $x^{*,d} = x^*$. Since the dynamics under each mechanism $k$ retain the structure of~\eqref{eq:vector_model}, with $M^k(x)$ Metzler and irreducible  for all $x\in[0,1]^n$, Theorem~\ref{theo:diffused-equilibrium}  applies with $M$, $x^*$ replaced by $M^k$, $x^{*,k}$. Hence, the  constraints $\rho(M^k(x^{*,k}))>1$ and~\eqref{eq:hyp1} in \eqref{eq:opt-constant} guarantee that $(a^{\dagger,k}, d^{\dagger,k})$ is a well-defined, unique, and globally asymptotically stable equilibrium.}


\MA{Let $\mathcal U^k(C):= \left\{ u^k\in U^k: \ \mathbf 1^\top u^k\le C,\ \rho(M^k(x^{*,k}))>1 \right\}$ denote the set of feasible constant controls.}

\begin{lemma}\label{lem:exist-constant-control}
    \MA{For each $k \in \{o,a,d\}$, if $\mathcal U^k(C)$ is not empty, then there exists at least one $\bar{u}^k \in \mathcal{U}^k(C)$ solving~\eqref{eq:opt-constant}.}
\end{lemma}
\begin{proof}
    \MA{The equilibrium $(a^{\dagger,k},d^{\dagger,k})$ depends continuously on the constant control $u^k$. The cost functional is therefore continuous in $u^k$ as a finite sum of continuous functions. Since $\mathcal U^k$ is compact, the Weierstrass theorem guarantees that $J$ attains a minimum on $\mathcal U^k(C)$.}
\end{proof}

\subsection{Model Predictive Control Formulation}
While constant control provides a baseline strategy, real-world interventions often require dynamic adjustments. To this end, we adopt a Model Predictive Control (MPC) framework over a finite horizon $N>0$. 
At each time $t$, the MPC approach solves a finite-horizon optimal control problem: given the current state, it computes a sequence of control inputs $u(\cdot|t)$ that minimizes a cumulative cost while ensuring feasibility of the system dynamics and compliance with budget constraints. 

\MA{Let $z(t):=(a(t),d(t),x(t))\in[0,1]^{3n}$ denote the aggregate state, and $z(\ell|t)$ its predicted value at step $\ell$ within the  horizon given the state at time $t$. For each $k\in\{o,a,d\}$, let $F^k:[0,1]^{3n}\times\mathcal{U}^k(C)\to[0,1]^{3n}$ denote the map collecting the  right-hand sides of~\eqref{eq:vector_model} under policy $k$, i.e., incorporating~\eqref{eq:opinion-control}, \eqref{eq:beta-control}, or~\eqref{eq:delta-control} in place of the corresponding uncontrolled terms. The MPC problem reads
\begin{equation}\label{eq:MPC}
\mspace{-5mu}\begin{aligned}
	\min_{U^k(t)}  \quad & \sum_{\ell=0}^{N-1} \sum_{i=1}^n \left[Q_i^a(a_i(\ell|t)\!-\!a_i^{\dagger,k})^2\!+\!Q_i^d(d_i(\ell|t)\!-\!d_i^{\dagger,k})^2 \right.\\[-2ex]
    &\mspace{200mu} \left. \!+\!L_i(u_i^k(\ell|t)\!-\!\bar u_i^k)^2 \right] \\
	\text{s.t.} \quad &  u^k(\ell|t) \in \mathcal{U}^k(C), \quad \forall\, \ell \in [0,N-1], \\
    &\, z(0|t) = z(t), \\
    &\, z(\ell+1|t) = F^k(z(\ell|t), u^k(\ell|t)), \quad \forall\, \ell \in [0,N-1],\\
    &\, z(N|t) = z^{\dagger,k},
\end{aligned}
\end{equation}
where $U^k(t) = \{ u^k(0|t), \dots, u^k(N-1|t) \}$ is the sequence of controls over the prediction horizon, and $z^{\dagger,k}$ is the equilibrium associated with the optimal constant policy $\bar u^k$ obtained from~\eqref{eq:opt-constant}, which by construction satisfies the $\rho(M^k(x^{*,k})) > 1$ and \eqref{eq:hyp1}.}
In practice, the control is implemented in a receding horizon fashion: at each time step only the first input $u^k(t)=u^{*,k}(0|t)$ is applied, and the optimization is repeated at the next step.

\subsection{Recursive Feasibility and Stability}
The next result establishes recursive feasibility of the MPC problem \eqref{eq:MPC}, guaranteeing that a feasible solution exists at each time step, and that the solution generated asymptotically converges to the desired equilibrium, \MA{following steps in \cite{ravazzi2026optimal}}.
We first clarify the definition of a feasible problem.
\begin{definition}
		Given $z(0)\in[0,1]^{3n}$ and $k\in\{o,a,d\}$, the problem~\eqref{eq:MPC} is \emph{feasible} if there exists \MA{$u^k(t)$ satisfying the constraints in~\eqref{eq:MPC} and driving $z(t)$ to $z^{\dagger,k}$ in $N$ steps}.
\end{definition}


\begin{proposition}\label{prop:feasibility}
	For each $k\in\{o,a,d\}$, if the problem ~\eqref{eq:MPC} is feasible at $t=0$, then 
    \begin{itemize}
        \item[(i)] it remains feasible for all $t \geq 0$;
        \item[(ii)] the corresponding MPC trajectory converges asymptotically to $z^{\dagger,k}$.
    \end{itemize}
\end{proposition}
\begin{proof}
    \MA{(i) Let $U^{*,k}(t) = \{ u^{*,k}(0|t), \dots, u^{*,k}(N-1|t) \}$ be an optimal feasible solution at time $t$. At time $t+1$, we construct the candidate sequence $\tilde{U}^k(t+1) = \{ u^{*,k}(1|t), \dots, u^{*,k}(N-1|t), \bar{u}^k \}$, given $\bar{u}^k$ from \eqref{eq:opt-constant}. By construction, all control constraints are satisfied. 
    Hence, $\widetilde U^k(t+1)$ is a feasible candidate solution at time $t+1$. Then, feasibility at time $t$ implies feasibility at time $t+1$. Since the problem is feasible at $t=0$, recursive feasibility follows for all $t\geq0$.}

    \MA{(ii) The optimal control sequence $U^{*,k}(t)$ has cost $J^*(t)$ at time $t$, while the candidate sequence at $t+1$, $\tilde{U}^k(t+1)$ has cost
    \begin{align*}
        \tilde{J}(t+1) &= J^*(t) - \ell(z(0|t), u^{k,*}(0|t)) + \ell(z^{\dagger,k}, \bar{u}^k) \\
        &= J^*(t)  - \ell(z(0|t), u^{k,*}(0|t))
    \end{align*} 
    where the last equality follows from $\ell(z^{\dagger,k}, \bar{u}^k) = 0$. By optimality, $J^*(t+1) \leq \tilde{J}(t+1) = J^*(t) - \ell(z(0|t), u^{k,*}(0|t)).$ Since $\ell \geq 0$, the sequence $J^*(t)$ is non-increasing  and bounded below by zero, hence convergent. This implies $\ell(z(t), u^{k,*}(0|t)) \to 0$ as $t \to \infty$, which by  positive definiteness of the cost gives  $z(t) \to z^{\dagger,k}$.}
\end{proof}

\section{Simulation Results}\label{sec:simulations}
\begin{figure*}
    \centering
    \includegraphics[width=16cm,height=4.2cm]{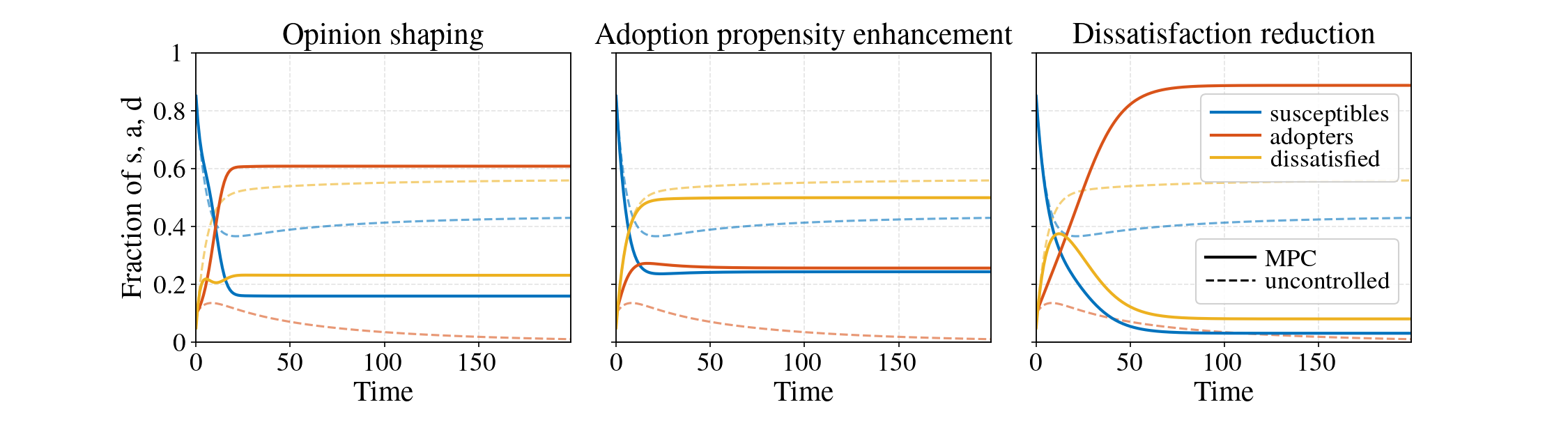}
    \caption{\MA{Aggregate controlled dynamics under the three MPC strategies: (left) opinion shaping \eqref{eq:opinion-control}, (center) adoption propensity enhancement \eqref{eq:beta-control}, and (right) dissatisfaction reduction \eqref{eq:delta-control}.}}
    \label{fig:mpc2}
\end{figure*}


We present numerical simulations illustrating the coupled adoption-opinion dynamics among $n=10$ interacting communities over two distinct networks, with randomly generated parameters and initial conditions corresponding to an early diffusion stage with few adopters and no dissatisfied individuals.
\MA{Figure \ref{fig:mpc2} compares the outcomes of three different control strategies, under identical initial conditions and a shared budget constraint. The optimal control inputs are computed using MATLAB's \texttt{fmincon} with the SQP algorithm over a receding horizon of $N=10$ steps. In each subplot, the dashed lines represent the uncontrolled trajectories, where adoption eventually vanishes, while the solid lines show the effect of the control policy.} The results suggest the following:
\MA{\begin{itemize}
    \item In this instance, (P1) drives the system toward a stable adoption level ($a^{\dagger,o} \approx 0.61$), with moderate dissatisfaction. Opinion shaping acts simultaneously on multiple channels: it increases the effective infection rate $\beta_i x_i$, accelerates recovery from dissatisfaction via $\gamma_i x_i$, and reduces the spontaneous churn $\theta_i(1-x_i)$.
    \item (P2) yields the least favorable trade-off in this scenario ($a^{\dagger,a} \approx 0.26$, $d^{\dagger,a} \approx 0.50$): propensity enhancement increases the flow to the adopter compartment, but leaves the dismissal rate $\delta_i$ unchanged.
    \item (P3) achieves the highest adoption level ($a^{\dagger,d} \approx 0.89$) with minimal dissatisfaction ($d^{\dagger,d} \approx 0.08$) in this instance, suggesting that retaining adopters by reducing dissatisfaction can be more effective than increasing new adoptions or shaping opinions for long-term diffusion.
\end{itemize}}
\MA{These findings should be interpreted with caution, as they are based on qualitative simulations with randomly generated parameters. A systematic quantitative comparison -- accounting for fairness, parameter sensitivity, and network structure -- is left for future work.}

\section{Conclusion}\label{sec:conclusion} 
We proposed an adoption-opinion model coupling compartmental adoption dynamics with opinion formation over a multilayer network. 
We established existence and uniqueness of the adoption-diffused equilibrium and derived stability conditions for both the adoption-free and adoption-diffused equilibria.
Building on this theoretical framework, we formulated an optimal control problem and compared three different MPC-based control strategies to maximize adoption. 

Future work will extend this analysis with systematic quantitative comparisons, considering uncertainty, \MA{performing sensitivity analysis}, and studying fairness and heterogeneity in the network structure. An important applied direction is the integration of empirical data, such as surveys or adoption records, to calibrate and validate the model in real-world settings like sustainable mobility or energy transition.
\section*{References}
\bibliographystyle{IEEEtran}
\bibliography{bib}


\end{document}